\documentclass{gen-j-l}
\usepackage{amssymb}
\usepackage{amsfonts}

\newtheorem{theorem}{Theorem}[section]
\newtheorem{lemma}[theorem]{Lemma}
\theoremstyle{definition}

\newtheorem{example}[theorem]{Example}

\theoremstyle{remark}
\newtheorem{remark}[theorem]{Remark}
\numberwithin{equation}{section}
\theoremstyle{plain}

\newtheorem{corollary}{Corollary}

\newtheorem{proposition}{Proposition}

\copyrightinfo{2001}{enter name of copyright holder}
\input{tcilatex}

\begin{document}
\title[Semirigid]{Semirigid GCD domains II}
\author{M. Zafrullah}
\address{Department of Mathematics, Idaho State University, Pocatello, 83209
ID}
\email{mzafrullah@usa.net}
\thanks{}
\subjclass[2000]{Primary 13A05, 13F15; Secondary 13G05}
\date{July 27, 2001}
\keywords{Semirigid, GCD domain, VFD, Pre-Schreier, HoFD, semi $t$-pure, SPD}
\dedicatory{Dedicated to the memory of Paul Cohn}

\begin{abstract}
Let $D$ be an integral domain with quotient field $K,$ throughout$.$ Call
two elements $x,y\in D\backslash \{0\}$ $v$-coprime if $xD\cap yD=xyD.$ Call
a nonzero non unit $r$ of an integral domain $D$ rigid if for all $x,y|r$ we
have $x|y$ or $y|x.$ Also call $D$ semirigid if every nonzero non unit of $D$
is expressible as a finite product of rigid elements. We show that a
semirigid domain $D$ is a GCD domain if and only if $D$ satisfies $\ast :$
product of every pair of non-$v$-coprime rigid elements is again rigid. Next
call $a\in D$ a valuation element if $aV\cap D=aD$ for some valuation ring $%
V $ with $D\subseteq V\subseteq K$ and call $D$ a VFD if every nonzero non
unit of $D$ is a finite product of valuation elements. It turns out that a
valuation element is what we call a packed element: a rigid element $r$ all
of whose powers are rigid and $\sqrt{rD}$ is a prime ideal. Calling $D$ a
semi packed domain (SPD) if every nonzero non unit of $D$ is a finite
product of packed elements, we study SPDs and explore situations in which an
SPD is a semirigid GCD domain.
\end{abstract}

\maketitle

\section{Introduction}

Let $D$ be an integral domain with quotient field $K.$ Some recent research
is a treasure trove of new ideas that are linked to some old ideas in an
uncanny fashion. Some of these are concepts such as a homogeneous element of
Chang \cite{[Ch]}, one that belongs to a unique maximal $t$-ideal and a
valuation element of Chang and Reinhart \cite{[CR]}, i.e. an element $a$
such that $aV\cap D=aD$ for some valuation ring $V$ with $D\subseteq
V\subseteq K.$ It is a beautiful notion and its properties are linked with
the so called rigid elements that I knew and worked with, long ago. The aim
of this note is to highlight the connections, smooth out the kinks caused by
changes in terminology and to present some new results.

Chang \cite{[Ch]} calls a domain a HoFD if every nonzero non unit of $D$ is
expressible as a product of mutually $t$-comaximal homogeneous elements and
says HoFDs were first studied in \cite{[AMZ]}, of course with different
terminology. (A homogeneous element was "$t$-pure" and a HoFD was a semi $t$%
-pure domain.) According to \cite[Corollary 1.2]{[CR]} a valuation element $%
a $ of a domain $D$ has the property that for all $x,y|a$ we have $x|y$ or $%
y|x.$ This makes $a$ a rigid element of Cohn \cite{[Coh rigid]}. But the
valuation elements of \cite{[CR]} beat Cohn's rigid elements by a mile in
properties. To be exact, let's call an element $r$ of $D$ rigid if $r$ is a
nonzero non-unit such that for all $x,y|r$ we have $x|y$ or $y|x.$ Then a
valuation element $a,$ in a VFD, is rigid such that every power of $a$ is
rigid and $\sqrt{a}$ is a prime \cite[Corollaries 1.2, 1.11]{[CR]}, here $%
\sqrt{a}$ represents the radical of he ideal $(a).$ Let's call an element
with these properties a packed element. Let's also call $D$ semirigid
(resp., semi homogeneous, semi packed) if every nonzero non unit of $D$ is
expressible as a finite product of rigid (resp., homogeneous, packed)
elements of $D.$ The trouble with the semirigid (resp., semi homogeneous)
domains is that they are very general. For example every irreducible
element, i.e., a nonzero non-unit $a$ such that $a=\alpha \beta \Rightarrow $
$\alpha $ is a unit or $\beta $ is, is rigid. But the atomic domains, i.e.,
domains whose nonzero non-units are expressible as finite products of
irreducible elements, often have little or no form of uniqueness of
factorization \cite{[AAZ]}. For example, in $D=F[[X^{2},X^{3}]],$ that is
Noetherian and hence atomic, the elements $X^{2}$ and $X^{3}$ are
irreducible, and $(X^{2})^{3}=(X^{3})^{2}=X^{6}.$ That is $X^{6}$ has two
distinct factorizations. On the other hand, as we shall show, semi
homogeneous domains are actually HoFDs and so do have a sort of uniqueness
of factorization, but only just.

One way of getting such wayward concepts to deliver unique factorization of
some sort is to bring in a somewhat stronger notion of coprimality and some
conditions. Call two elements $a,b$ of a domain $D.$ $v$-coprime if $aD\cap
bD=abD$. Obviously $a,b$ are $v$-coprime if and only if $(a,b)^{-1}=D,$ if
and only, if $((a,b)^{-1})^{-1}=(a,b)_{v}=D,$ where $A\mapsto
A_{v}=(A^{-1})^{-1}$ is the usual star operation called the $v$-operation on 
$F(D),$ the set of nonzero factional ideals of $D.$ The notion of $v$%
-coprimality has been discussed in detail in \cite{[Z v-cop]}, where it is
shown, in somewhat general terms, that if, for $a,b,c\in D\backslash \{0\},$ 
$(a,b)_{v}=D$ and $a|bc,$ then $a|c.$ It was also shown in \cite[Proposition
2.2]{[Z v-cop]} that for $r_{1},...,r_{n},x\in D\backslash \{0\}$ $%
(r_{1}...r_{n},x)_{v}=D$ if and only if $(r_{i},x)_{v}=D.$ Let's call two
homogeneous (resp., rigid) elements $a,b$ similar, denoted $a\symbol{126}b,$
if $(a,b)_{v}\neq D$. We plan to show that a semi homogeneous domain is a
"HoFD" because the product of every pair of similar homogeneous elements of $%
D$ is again a homogeneous element, of $D,$ similar to them. We also show in
Section 2 that a semirigid domain is a semirigid GCD domain if and only if
the product of each pair of non-$v$-coprime rigid elements is rigid and give
examples to show that the product of two non-$v$-coprime rigid elements may
not be rigid. We shall also give examples to show that a homogeneous element
may not be rigid and a rigid element may not be homogeneous. On the semi
packed front, we establish semi packed domains (SPDs) as a concept parallel
to VFDs in section 3 and give at least one example of an SPD that is not a
semirigid GCD domain. We also explore the conditions that make an SPD (resp
a VFD) into a semirigid GCD domain.

It seems best to give the reader an idea of the $v$- and the $t$-operations
and some related concepts that we shall have the occasion to use. For $I\in
F(D),$ the set $I^{-1}=\{x\in K|xI\subseteq D\}$ is again a fractional ideal
and thus the relation $v$: $I\mapsto I_{v}$ is a function on $F(D).$ This
function is called the $v$-operation on $D.$ Similarly the relation $t$: $%
I\mapsto I_{t}=\cup \{F_{v}|$ $0\neq F$ is a finitely generated subideal of $%
I\}$ is a function on $F(D)$ and is called the $t$-operation on $D.$ These
are examples of the so called star operations. The reader may consult Jesse
Elliott's book \cite{[E]} for these operations. A fractional ideal $I$ is
called a $v$-ideal (resp. a $t$-ideal) if $I_{v}=I$ (resp., $I_{t}=I).$The
rather peculiar definition of the $t$-operation allows one to use Zorn's
Lemma to prove that each integral domain that is not a field has at least
one integral $t$-ideal maximal among integral $t$-ideals. This maximal $t$%
-ideal is prime and every proper, integral $t$-ideal is contained in at
least one maximal $t$-ideal. A minimal prime of a $t$-ideal is a $t$-ideal
and thus every height one prime is a $t$-ideal. The set of all maximal $t$%
-ideals of a domain $D$ is denoted by $t$-$Max(D).$ It can be shown that $%
D=\cap _{M\in t\text{-}Max(D)}D_{M}.$ A fractional ideal $I$ is said to be $%
t $-invertible if $(II^{-1})_{t}=D.$ A domain in which every nonzero
finitely generated ideal is $t$-invertible is called a Prufer $v$%
-multiplication domain (PVMD), a Prufer domain is a PVMD with every nonzero
ideal a $t$-ideal. Griffin \cite{[G]} showed that $D$ is a PVMD if and only
if $D_{M}$ is a valuation domain for each maximal $t$-ideal $M$ of $D.$
Given any domain $D$ the set $t$-$inv(D)$ of all $t$-invertible fractional $%
t $-ideals of $D$ is a group under the $t$-operation (i.e., $I\times
_{t}J=(IJ)_{t}).$ The group $t$-$inv(D)$ has the group $P(D)$ of nonzero
principal fractional ideals as its subgroup. The $t$-class group of $D$ is
the quotient group $Cl_{t}(D)=$ $t$-$inv(D)/P(D).$ What makes this group
interesting is that if $D$ is a Krull domain $Cl_{t}(D)$ is the divisor
class group of $D$ and if $D$ is a Prufer domain, $Cl_{t}(D)$ is the ideal
class group of $D$. Of interest for this note is the fact that a PVMD $D$ is
a GCD domain if and only if $Cl_{t}(D)$ is trivial. This group was
introduced in \cite{[B]}.

Next, call an element $a\in D\backslash \{0\}$ primal if for all $b,c\in
D\backslash \{0\}$ $a|bc$ implies that $a=rs$ where $r|b$ and $s|c.$ A
domain all of whose nonzero elements are primal is called a pre-Schreier
domain and an integrally closed pre-Schreier domain was called a Schreier
domain in \cite{[Coh bezout]}. Note that if $D$ is pre-Schreier then $%
Cl_{t}(D)$ is trivial, \cite{[B]}. Call a nonzero element $p$ of $D$
completely primal if every factor of $p$ is again primal. A power of a prime
element is an example of a completely primal element. According to Cohn \cite%
{[Coh bezout]}, if $S$ is a set multiplicatively generated by completely
primal elements of an integrally closed domain $D$ such that $D_{S}$ is a
Schreier domain, then $D$ is a Schreier domain. This Theorem is usually
referred to as: Cohn's Nagata type Theorem for Schreier domains. Our
terminology is standard or is explained at the point of entry.

\section{Semirigid GCD domains}

Let's note that for a finitely generated nonzero ideal $I=(x_{1},...,x_{n})$
we have $I_{v}=I_{t},$ so $x_{1},...,x_{n}$ being $v$-coprime (i.e., $%
(x_{1},...,x_{n})_{v}=D$) is the same as $x_{1},...,x_{n}$ being $t$%
-comaximal (i.e., $(x_{1},...,x_{n})_{t}=D$), which boils down to: $%
x_{1},...,x_{n}$ do not share any maximal $t$-ideal. We also note that $a$
is a homogeneous element if $aD$ is a $t$-homogeneous ideal in the sense of 
\cite{[AZ starsh]} and, sort of, following the convention of \cite{[AZ
starsh]} we shall denote by $M(a)$ the maximal $t$-ideal containing the
homogeneous element $a.$ Indeed we have $M(a)=\{x\in D|$ $(x,a)_{v}\neq D\}$
(cf. \cite[(2) Proposition 1]{[AZ starsh]}). The following two results can
be proved using Theorem 3.1 of \cite{[AMZ]}, but due to the change of
terminology, mentioned in the introduction, it seems safe to redo them here.
(I plan to address the change of terminology later and suggest a way to
patch things up.)

\begin{lemma}
\label{Lemma L1}Let $a$ and $b$ be two homogeneous elements of $D$ then $%
(a,b)_{v}\neq D$ if and only if $(a,b)$ is contained in the same maximal $t$%
-ideal if and only if $ab$ is a homogeneous element.
\end{lemma}

\begin{proof}
Let $b$ be a homogeneous element belonging to the maximal $t$-ideal $P.$ For
any nonzero finitely generated ideal $A,$ $(A,b)_{v}\neq D$ implies that
that $A\subseteq P.$ This is because $(A,b)_{v}\neq D$ implies $(A,b)$ has
to be contained in some maximal $t$-ideal and $P$ is the only maximal $t$%
-ideal that contains $b.$ So $A\subseteq P.$ Now $(a,b)_{v}\neq D$ implies
that $a,b$ both belong to the same maximal $t$-ideal say $P.$ Next note that 
$x\in M(a)\Leftrightarrow (x,a)_{v}\neq D.$ So $x\in M(a)$ implies $x$
belongs to $P.$ Thus $M(a)=P$ and similarly $M(b)=P$ forcing $M(a)=M(b).$
Suppose $ab$ belongs to a maximal $t$-ideal $P.$ Then $a\in P$ or $b\in P.$
If $a\in P,$ then $M(a)=P.$ But as $(b,a)_{v}\neq D,$ $M(a)=M(b)$ whence $ab$
is a homogeneous element, as $P(a)=P(b)$ is the only maximal $t$-ideal
containing $ab$. Finally if $ab$ is $t$-homogeneous then, by definition, $%
(a,b)_{v}\neq D.$
\end{proof}

\begin{proposition}
\label{Proposition P1} An integral domain $D$ is a HoFD if and only if $D$
is a semi homogeneous domain.
\end{proposition}

\begin{proof}
Suppose that $D$ is a semi homogeneous domain. Lemma \ref{Lemma L1} shows
that the product of every pair of similar homogeneous elements of $D$ is
homogeneous. Let $x=h_{1}h_{2}...h_{n}$ where each of $h_{i}$ is a
homogeneous element. Now $M_{1},...M_{r}$ be the set of distinct maximal $t$%
-ideals containing $h$. Let $H_{j}=\Pi h$ where $h$ ranges over $h_{i}\in
M_{j}.$ By Lemma \ref{Lemma L1}, $H_{j}$ are homogeneous and mutually $t$%
-comaximal. Thus we have $x=\Pi _{j=1}^{r}H_{j}$ where $H_{i}$ are mutually $%
v$-coprime homogeneous. The converse is obvious.
\end{proof}

It was shown in \cite{[Z semi]} that if a nonzero non unit $x$ in a GCD
domain is expressible as a finite product of rigid elements, then $x$ is
uniquely expressible as a product of finitely many mutually coprime rigid
elements. Thus showing that in a semirigid GCD domain every nonzero non unit 
$x$ is expressible uniquely as a product of mutually coprime elements. So a
valuation ring $V$ of any rank is an example of a semirigid GCD domain and
so is a polynomial ring over $V$. Griffin, in \cite{[Gr]}, called a domain $%
D $ an Independent Ring of Krull type (IRKT) if $D$ has a family of prime
ideals $\{P_{\alpha }\}_{\alpha \in I}$ such that (a) $D_{P_{\alpha }}$ is a
valuation domain for each $\alpha \in I,$ (b) $D=\cap _{\alpha \in
I}D_{P_{\alpha }}$ is locally finite and (c) No pair of distinct members of $%
\{P_{\alpha }\}_{\alpha \in I}$ contains a nonzero prime ideal. It was shown
in Theorem 5 of \cite{[Z semi]} that a semirigid GCD domain is indeed an
IRKT. Also, it was shown in Theorem B of \cite{[Z rigid]} that a GCD IRKT is
a semirigid GCD domain. Later, a domain satisfying only (b) and (c) above,
requiring that $P_{\alpha }$ are maximal $t$-ideals, was called in \cite{[AZ
splittin]} a weakly Matlis domain. An IRKT is a PVMD, \cite{[Gr]}. Also,
suppose that $D$ is a weakly Matlis GCD domain. Noting that a GCD domain is
a PVMD which makes localization at each maximal $t$-ideal a valuation domain
we have each of $D_{P_{\alpha }}$ a valuation domain, in the definition of a
weakly Matlis domain and making it an IRKT. Finally, a GCD IRKT is a
semirigid GCD domain, by Theorem B of \cite{[Z rigid]}. We now show that
introducing a simple property $\ast :$ for every pair $r,s$ of rigid
elements $(r,s)_{v}\neq D\Leftrightarrow rs$ is rigid, we can make a
semirigid domain into a semirigid GCD domain.

\begin{lemma}
\label{Lemma L2} Let $D$ be a semirigid domain with $\ast :$ for every pair $%
r,s$ of rigid elements $(r,s)_{v}\neq D\Leftrightarrow rs$ is rigid$.$ Then
the following hold. (1) Suppose that $r,s$ are two similar rigid elements.
Then $r$ and $s$ are comparable, i.e., $r|s$ or $s|r,$ (2) If $r$ is a rigid
element and $s,t$ are rigid elements, each similar to $r,$ then $s$ and $t$
are similar, (3) A finite product of mutually similar rigid elements is
rigid similar to each of the factors and (4) if a rigid element $r$ divides
a product $x=x_{1}x_{2}...x_{n}$ of mutually $v$-coprime rigid elements $%
x_{1},...,x_{n}$ then $r$ divides exactly one of the $x_{i},$ in a semirigid
domain with property $\ast .$
\end{lemma}

\begin{proof}
(1). Straightforward, as $rs$ is rigid.

(2). $r|s$ or $s|r$ and $r|t$ or $t|r.$ Four cases arise (i) $r|s$ and $r|t$ 
$\Rightarrow (s,t)_{v}\neq D,$ (ii) $r|s$ and $t|r\Rightarrow t|s$ (iii) $%
s|r $ and $r|t\Rightarrow s|t$ (iv) $s|r$ and $t|r\Rightarrow s|t$ or $t|s.$
In each case we have $s\symbol{126}t.$

(3). Suppose that $D$ is semirigid with the given property ($\ast $). Using
induction, one can show that in a semirigid domain with ($\ast ),$ a finite
product of mutually similar rigid elements is rigid. This is how it can be
accomplished: We know that the product of any two similar rigid elements is
rigid. Assume that we have established that the product of any set of $n$ of
rigid elements, $r_{1},r_{2},...,r_{n},$ similar to one of them and hence,
by (2), similar to each other, is rigid. Let $\boldsymbol{r~}%
=~r_{1}r_{2}...r_{n}$ and let $s$ be a rigid element similar to, one and
hence, each of $r_{i}$ and hence to $\mathbf{r.}$ But then by $\ast ,$ $%
\mathbf{r}s$ is rigid.

(4). Note that $(r,x)_{v}=rD\neq D,$ because $r|x.$ So $r$ cannot be $v$%
-coprime to each of $x_{i},$ \cite[Proposition 2.2]{[Z v-cop]}. Now, say, $r$
is non-$v$-coprime to $x_{i},x_{j}$ for $i\neq j.$ But then, by (2), $x_{i}%
\symbol{126}x_{j}$ which is impossible because $(x_{i},x_{j})_{v}=D.$ So $r$
is non-$v$-coprime to exactly one of $x_{i},$ say $x_{k}.$ Now as $D$ has
the property $\ast $ and as $r$ and $x_{k}$ are rigid, one of them divides
the other. But since $r|x$ already, we conclude that $r|x_{k}$.
\end{proof}

\bigskip

\begin{proposition}
\label{Proposition P1A} Let $D$ be a semirigid domain. Then every nonzero
non unit of $D$ is either rigid or can be written uniquely as a product of
finitely many mutually $v$-coprime rigid elements if and only if $\ast $:
for every pair $r,s$ of rigid elements $(r,s)_{v}\neq D\Leftrightarrow rs$
is rigid holds.
\end{proposition}

\begin{proof}
Let $x=r_{1}r_{2}...r_{n}$ be a nonzero non unit of $D.$ Pick $r_{1}$ and
collect all the rigid factors, from $r_{i},$ ($i=1,...,n$), that are similar
to $r_{1}.$ Next suppose that by a relabeling we can write $%
x=r_{1}r_{2}...r_{s_{1}}r_{s_{1}+1}...r_{n}$ where $r_{i}$ ($i=1,...,s_{1})$
are all the rigid factors of $x$ that are similar to $r_{1}.$ Set $\mathbf{r}%
_{1}=r_{1}r_{2}...r_{s_{1}}$. Note that since, by the procedure, each of $%
r_{i}$ ($i=1,...,s_{1})$ is $v$-coprime to each of $r_{i_{1}}$ ($%
i_{1}=s_{1}+1,...,n)$ we conclude that $\mathbf{r}_{1}$ is $v$-coprime to
each of $r_{i_{1}}$ ($i_{1}=s_{1}+1,...,n$) and, of course, each of $%
r_{i_{1}}$ ($i_{1}=s_{1}+1,...,n$) $v$-coprime to $\mathbf{r}_{1}.$ Now
select all the rigid elements similar to $r_{s_{1}+1}$ and suppose that by a
relabeling we can write $%
r_{s_{1}+1}...r_{n}=r_{s_{1}+1}r_{s_{1}+2}...r_{s_{2}}...r_{n}$, where $%
r_{j} $ ($j=s_{1}+1,...,s_{2})$ are similar to $r_{s_{1}+1}.$ Set $\mathbf{r}%
_{2}=r_{s_{1}+1}r_{s_{1}+2}...r_{s_{2}}.$ By, Lemma \ref{Lemma L2}, $\mathbf{%
r}_{2}$ is rigid. Since $r_{s_{1}+1},r_{s_{1}+2},...,r_{s_{2}}$ are each $v$%
-coprime to $\mathbf{r}_{1},$ and so is their product, we conclude that $%
\mathbf{r}_{1}$ and $\mathbf{r}_{2}$ are $v$-coprime rigid elements. Thus $x=%
\mathbf{r}_{1}\mathbf{r}_{2}r_{s_{2}+1}...r_{n}$ and continuing in this
manner we can write $x=\mathbf{r}_{1}\mathbf{r}_{2}\mathbf{...r}_{m}$ where $%
\mathbf{r}_{i}$ are mutually $v$-coprime rigid elements.

Now let $x=\mathbf{r}_{1}\mathbf{r}_{2}\mathbf{...r}_{m}$ be a product of
mutually $v$-coprime rigid elements in a domain $D$ with property $\ast .$
Also let $x=\mathbf{s}_{1}\mathbf{s}_{2}\mathbf{...s}_{n}.$ We claim that
each of the $\mathbf{r}_{i}$ is an associate of exactly one of the $\mathbf{s%
}_{j}$ and hence $m=n.$ For this note that by (4) of Lemma \ref{Lemma L2}, $%
\mathbf{r}_{1}|\mathbf{s}_{1}\mathbf{s}_{2}\mathbf{...s}_{n}$ implies that $%
\mathbf{r}_{1}$ divides exactly one one of the $\mathbf{s}_{j},$ say $%
\mathbf{s}_{1},$ by a relabeling. But then, considering $\mathbf{s}_{1}|$ $%
\mathbf{r}_{1}\mathbf{r}_{2}\mathbf{...r}$ and noting that $\mathbf{s}_{1}%
\symbol{126}\mathbf{r}_{1}\mathbf{\ }$we conclude that $\mathbf{s}_{1}|%
\mathbf{r}_{1}.$ This leaves us with $\mathbf{r}_{2}\mathbf{...r}_{m}=%
\mathbf{s}_{2}\mathbf{...s}_{n}.$ Now eliminating one by one, in this
manner, and noting that $\mathbf{r}_{i},$ $\mathbf{s}_{j}$ are non units
will take us to the conclusion, eventually.

Conversely let $D$ be a semirigid domain in which every nonzero non unit is
either a rigid element or is uniquely expressible as a finite product of
mutually $v$-coprime rigid elements and consider $x=rs$ where $r$ and $s$
are any two similar rigid elements. If in each case $rs$ is rigid, we are
done. To ensure that there is no other possibility we proceed as follows.
Because $x$ is expressible as a product of finitely many mutually $v$%
-coprime rigid elements we can write $rs=r_{1}...r_{n}.$ By (4) of Lemma \ref%
{Lemma L2}, each of $r,s$ divides exactly one of the $r_{i},$ so $n\leq 2.$
So let $rs=r_{1}r_{2},$ where $r|r_{1}$ and $s|r_{2}.$ Now this too is
impossible because, by assumption, $r$ and $s$ are non-$v$-coprime while $%
r_{1}$ and $r_{2}$ are $v$-coprime and obviously a pair of $v$-coprime
elements (such as $r_{1},r_{2})$ cannot have factors like $r|\acute{r}_{1}$
and $s|r_{2}$ with $(r,s)_{v}\neq D.$ (For this note that $r|\acute{r}_{1}$
and $s|r_{2}$ implies that $(\acute{r}_{1},r_{2})\subseteq (r,s)$ forces $D=(%
\acute{r}_{1},r_{2})_{v}\subseteq (r,s)_{v}\neq D,$ which is a
contradiction.)
\end{proof}

\begin{theorem}
\label{Theorem T1} A semirigid domain with property $\ast $ is a GCD domain.
\end{theorem}

\begin{proof}
Let $r$ be a rigid element and $x$ be any nonzero non unit of $D.$ By
Proposition \ref{Proposition P1A} $x=x_{1}x_{2}...x_{n}$ where $x_{i}$ are
mutually $v$-coprime rigid elements. Two cases arise: (a) $r$ is $v$-coprime
to each of $x_{i}$ and (b) $r$ is non-$v$-coprime to one (and exactly one)
say $x_{1}$ of the $x_{i}.$ (By Lemma \ref{Lemma L2} $r$ cannot be non-$v$%
-coprime to more than one). Now in the presence of property $\ast ,$ $r|x_{1}
$ or $x_{1}|r.$ In case (a) $(r,x)_{v}=D$ and in case (b) if $r|x_{1}$ then $%
r|x$ and so $(r,x)_{v}=rD$. Finally if $x_{1}|r$, then since $(r,x_{i})_{v}=D
$ for $i=2,...,n,$ we must have $(r/x_{1},x_{i})_{v}=D$ for $i=2,...,n.$ But
that gives $(r/x_{1},x_{2}...x_{n})_{v}=D,$ forcing $(r,x)_{v}=x_{1}D.$
Throwing in the case when $x$ is a unit we conclude that if $r$ is a rigid
element and $x$ any nonzero element of a domain with property $\ast $, then $%
(r,x)_{v}=hD$ a principal ideal. We next show that $r$ is primal. For this
we let $r|xy$ for some $x,y\in D\backslash \{0\}.$ Letting $x=sx_{1}$ where $%
s$ is such that $(r,x)_{v}=sD$ we have $r=st$ where $(t,x_{1})_{v}=D.$ Now $%
st|sx_{1}y$ leads to $t|x_{1}y$, which in turn leads to $t|y$ because $%
(t,x_{1})_{v}=D.$ Whence $r=st$ where $s|x$ and $t|y.$ Since the choice of $%
x,y$ was arbitrary, we conclude that every rigid element $r$ of $D$ is
primal. Combine the information that products of primal element are primal (%
\cite{[Coh bezout]}) with the fact that $D$ is semirigid to conclude that $D$
is indeed pre-Schreier. Now in a pre-Schreier domain a rigid element is
homogeneous \cite[Theorem 2.3]{[AMZ]} and so belongs to a unique maximal $t$%
-ideal $P(r)=\{x\in D|$ $(r,x)_{v}\neq D\}.$ Once we have shown that every
rigid element of a semirigid domain $D$ is homogeneous we can conclude that $%
D$ is a HoFD. Next, in view of proof of (3) of Theorem 3.1 of \cite{[AMZ]},
one can say that if $D$ is a HoFD and $M$ a maximal $t$-ideal of $D$ and if $%
0\neq xD_{M}\subseteq MD_{M}$ then $xD_{M}\cap D=x^{\prime }D$ where $%
x^{\prime }$ is a homogeneous element contained in $M$ such that $%
xD_{M}=x^{\prime }D_{M}.$ So let $M$ be a maximal $t$-ideal of $D$ and let $%
0\neq hD_{M},kD_{M}\subseteq MD_{M}.$ Then we have $hD_{M}\cap D=h^{\prime }D
$ and $kD_{M}\cap D=k^{\prime }D$ where $h^{\prime }$ and $k^{\prime }$ are
homogenous elements contained in $M$ and hence similar. But in our domain $D$
with property $\ast ,$ $h^{\prime },k^{\prime }$similar means $h^{\prime
}D\subseteq k^{\prime }D$ or $k^{\prime }D\subseteq h^{\prime }D.$ Extending
these to $D_{M}$ we get $hD_{M}=$ $h^{\prime }D_{M}\subseteq k^{\prime
}D_{M}=kD_{M}$ or $kD_{M}=k^{\prime }D_{M}\subseteq h^{\prime }D_{M}=hD_{M}$%
. Thus every pair of principal ideals $hD_{M},kD_{M}$ of $D_{M}$ is
comparable and $D_{M}$ is a valuation domain. Now as the choice of $M$ was
arbitrary we conclude that $D_{M}$ is a valuation domain for every maximal $t
$-ideal of $D$ and $D$ is a PVMD. But a pre-Schreier PVMD is a GCD domain 
\cite[Corollary 1.5]{[BZ]}.
\end{proof}

\begin{corollary}
\label{Corollary C2} For an integral domain $D$ the following are equivalent.
\end{corollary}

(1) $D$ is a semirigid domain with property $\ast ,$

(2) $D$ is a semirigid GCD domain,

(3) $D$ is a semi homogeneous GCD domain,

(4) $D$ is a HoFD PVMD,

(5) $D$ is a weakly Matlis GCD domain,

(6) $D$ is a GCD VFD,

(7) $D$ is a UVFD,

(8) $D$ is a VFD such that product of every pair of non-coprime valuation
elements is again a valuation element,

(9) $D$ is a pre-Schreier semirigid domain such that products of pairs of
non-corime rigid elements are again rigid and

(10) $D$ is a semirigid PSP domain such that every element of $D$ is rigid
or is expressible as a product of mutually $v$-coprime rigid elements.

\begin{proof}
(1) $\Rightarrow $ (2). This follows from Theorem \ref{Theorem T1}.

(2) $\Rightarrow $ (3). Let $r$ be a rigid element of a GCD domain $D$ and
consider $P(r)=\{x\in D|(x,r)_{v}\neq D\}.$ By \cite[Lemma 1]{[Z semi]} $%
P(r) $ is a prime ideal. To see that $P(r)$ is a $t$-ideal note that because 
$D$ is a GCD domain, $(x,r)_{v}\neq D\Leftrightarrow x=a\rho $ where $\rho $
is a non unit factor of $r.$ Thus $x_{1},x_{2},...,x_{n}\in P(r)$ $%
\Rightarrow (x_{1},x_{2},...,x_{n})\subseteq (\rho _{1})$ where $\rho _{1}$
is a non unit factor of $r.$ But then, $x_{1},x_{2},...,x_{n}\in P(r)$ $%
\Rightarrow (x_{1},x_{2},...,x_{n})_{v}\subseteq P(r).$ Thus $P(r)$ is a $t$%
-ideal. Finally, let $M$ be a prime ideal properly containing $P(r)$ and let 
$y\in M\backslash P(r).$ By the definition of $P(r),$ $(y,r)_{v}=D.$ But
then $M_{t}=D$ and this shows that $P(r)$ is actually a maximal $t$-ideal.
Finally, using the definition, it can be easily established that for any
pair of rigid elements $r,s$ of $D,$ $P(r)=P(s)$ if and only if $r\symbol{126%
}s.$ Thus every rigid element, in a GCD domain $D,$ belongs to a unique
maximal $t$-ideal and hence is a homogeneous element. Consequently a
semirigid GCD domain is a semi homogeneous GCD domain.

(3) $\Rightarrow $ (4). Note that a semi homogeneous domain is a HoFD, by
Proposition \ref{Proposition P1} and, it is well known that, a GCD domain is
a PVMD.

(4) $\Rightarrow $ (5). Since a HoFD $D$ is a weakly Matlis domain with $%
Cl_{t}(D)=0$ \cite[Theorem 2.2]{[Ch]} or \cite[Theorem 3.4]{[AMZ]} and since
a PVMD $D$ with $Cl_{t}(D)=0$ is a GCD domain \cite[Proposition 2]{[B]}, we
conclude that a PVMD\ HoFD is a weakly Matlis GCD domain.

(5) $\Rightarrow $ (6). Note that Corollary 4.5 of \cite{[CR]} says that a
UVFD is a weakly Matlis GCD domain and as a UVFD is a VFD, in particular, we
have the conclusion.

(6) $\Rightarrow $ (7). This follows because, according to Theorem 4.2 of 
\cite{[CR]}, a PVMD VFD is a UVFD and so a GCD VFD is a UVFD.

(7) $\Rightarrow $ (8). Let $D$ be a UVFD then $D$ is a VFD. Take two
non-coprime valuation elements $u,v$ and using the fact that $uv$ is an
element of a UVFD write $uv=a_{1}a_{2}...a_{n}$ where $a_{i}$ are mutually
incomparable and hence mutually coprime. Now $u|a_{1}a_{2}...a_{n}$ implies
that $u=u_{1}u_{2}...u_{n}$ where $u_{i}|a_{i},$ because a VFD is Schreier 
\cite{[CR]}, see also Cohn \cite{[Coh bezout]}. We claim that exactly one of
the $u_{i}$ is non unit. For if say $u_{1}$ and $u_{2}$ are non units then
being factors of coprime elements $u_{1}$ and $u_{2}$ are incomparable and
this contradicts the fact that $u$ is a valuation element (cf. \cite[(2) of
Corollary 1.2]{[CR]}). So $u$ divides exactly one of the $a_{i}.$ Similarly $%
v$ divides exactly one of the $a_{i}.$ Next $u$ and $v$ cannot divide two
distinct $a_{i}$ for that would make $u,v$ coprime, which they are not. Now
suppose that $u|a_{1}$. Then $v=(a_{1}/u)a_{2}...a_{n}$ and $v$ cannot
divide any of $a_{2},...,a_{n}$ as that would make $v$ coprime with $u.$ So, 
$v$ must divide $(a_{1}/u).$ But then $1=(a_{1}/uv)a_{2}...a_{n},$ forcing $%
uv=a_{1}$ and forcing the conclusion that in the VFD $D$ the product of any
pair of non coprime valuation elements is again a valuation element.

(8) $\Rightarrow $ (9). Note that as each valuation element is rigid, a VFD
is semirigid. Since a VFD is Schreier we can say that $D$ is a pre-Schreier
semirigid domain. Also the product of two non coprime valuation elements
being a valuation element translates to the product of two non-$v$-coprime
rigid elements is rigid and that is the property $\ast .$

(9) $\Rightarrow $ (1). Follows directly, as in a pre-Schreier domain
non-coprime is the same as non-$v$coprime.

(2) $\Rightarrow $ (10). This follows directly.

(10) $\Rightarrow $ (1). Let $x=rs$ where $r,s$ are non-$v$-coprime rigid
elements. If $rs$ is not rigid, then $rs=t_{1}...t_{n}$ where $t_{i}$ are
mutually $v$-coprime rigid elements. Now $r$ cannot be non-$v$-coprime to
more than one of the $t_{i}.$ For if it were it would have non unit factors
common with $v$-coprime elements, forcing $r$ to be non rigid. Thus $r$
divides exactly one of the $t_{i}.$ Again $r$ and $s$ cannot divide two
distinct $t_{i}$ because then they would be $v$-coprime. This forces $rs$ to
divide exactly one of the $t_{i},$ leading to the conclusion that $D$ is a
semirigid domain in which the product of every pair of non-$v$-coprime rigid
elements is rigid.
\end{proof}

While Corollary \ref{Corollary C2} establishes that the older concept of
semirigid GCD domains of \cite{[Z semi]} is precisely the most recent
concept of UVFDs of \cite{[CR]}, it raises the following question.

Question \label{Question Q1} Must a Schreier semirigid domain be a semirigid
GCD domain?

This question becomes interesting in view of the fact that the authors of 
\cite{[CR]} ask a similar question: is a VFD a semirigid GCD domain?
(Actually they ask: Is a VFD a weakly Matlis GCD domain?) The other point of
interest is that, according to \cite{[Coh bezout]} an atomic Schreier domain
is a UFD. In fact, once we recall necessary terminology, we have the
following more general result.

\begin{proposition}
\label{Proposition P2} (cf., \cite{[AZ sch]}, Proposition 3.2). In an
integral domain with PSP property, every atom is a prime. Consequently an
atomic domain with PSP property is a UFD.
\end{proposition}

Here a polynomial $f(X)=\sum_{i=0}^{i=n}a_{i}X^{i}$ is primitive if the
coefficients $a_{i}$ of $f$ have no non unit common factor and $f$ is super
primitive if the coefficients $a_{0},...,a_{n}$ are $v$-coprime. Also a
domain $D$ has the PSP property if every primitive polynomial is super
primitive. Now as was indicated in \cite{[AZ sch]} a domain with PSP
property is much more general than a pre-Schreier domain. Lest hopes run too
high, we hasten to offer the following example of a semirigid Schreier
domain in which the product of two rigid elements is not rigid.

\begin{example}
\label{Example E1} Let $%
\mathbb{Z}
$ denote the ring of integers, let $%
\mathbb{Q}
$ be the ring of rational numbers and let $X,Y$ be two indeterminates over $%
\mathbb{Q}
.$ Construct the two dimensional regular local ring $R=%
\mathbb{Q}
\lbrack \lbrack X,Y]]$ and for $p$ a prime element set $D=%
\mathbb{Z}
_{(p)}+(X,Y)%
\mathbb{Q}
\lbrack \lbrack X,Y]].$ This ring $D$ is a semirigid Schreier domain with
two rigid elements $X,Y$ such that $(X,Y)_{v}\neq D,$ yet $XY$ is not rigid.
\end{example}

Illustration: Indeed $D$ is a quasi local ring with maximal ideal principal
and of course $X$ and $Y$ are divisible by every power of $p.$ That $D$ is
integrally closed follows from the fact that $D\subseteq 
\mathbb{Q}
\lbrack \lbrack X,Y]]$ which is integrally closed, that $%
\mathbb{Z}
_{(p)}$ is integrally closed and that $X$ and $Y$ are divisible by powers of 
$p$. For $D$ being Schreier let $S$ be multiplicatively generated by $p.$
Then $S$ is generated by completely primal elements and $D_{S}$ is a UFD.
Hence $D$ is Schreier, by Cohn's Nagata type Theorem. Now look at $X.$ Every
factor of $X$ of the form $p^{r}$ or $X/p^{s}.$ So any pair of factors of $X$
is one of the forms: $(p^{r},p^{s}),(p^{r},X/p^{s}),(X/p^{r},X/p^{s}),$ $%
r,s\geq 0,$ and in each case one of them divides the other. Same with $Y.$
Now $(X,Y)_{v}\neq D,$ because $p|X,Y.$ So $X$ and $Y$ are non-$v$-coprime
rigid elements of $D.$ Yet $XY$ cannot be, because $X$ does not divide $Y.$
Finally, as $%
\mathbb{Q}
\lbrack \lbrack X,Y]]$ is a UFD with each prime an element $f(X,Y)$ such
that $f(0,0)=0$ we conclude that for each prime $f$ of $%
\mathbb{Q}
\lbrack \lbrack X,Y]],$ $f/p^{r}$ is rigid in $%
\mathbb{Z}
_{(p)}+(X,Y)%
\mathbb{Q}
\lbrack \lbrack X,Y]].$ But then a typical nonzero non unit element of $%
\mathbb{Z}
_{(p)}+(X,Y)%
\mathbb{Q}
\lbrack \lbrack X,Y]]$ is expressible as a finite product of rigid elements.

\begin{remark}
\label{Remark R1} The above example has often appeared, in various
capacities, in papers in which Dan Anderson and I have been coauthors, see
e.g. \cite[page 344]{[AMZ]}, \cite{[AZ splittin]}, though the current
application is much more elaborate and hence different. Also, as we shall
see in the next section that, the above example can have another
interpretation/application.
\end{remark}

It was shown in the proof of (2) $\Rightarrow $ (3) of Corollary \ref%
{Corollary C2} that a rigid element is a homogeneous element. However a
rigid element may not generally be a homogeneous element. For an atom is
rigid but, say, in a Krull domain with torsion divisor class group an atom
can be in more than one height one prime ideals which can be shown to be
maximal $t$-ideals. (For a concrete example $%
\mathbb{Z}
\lbrack \sqrt{-5}]$ is a Dedekind domain in which $3$ is well known to be an
irreducible element, but $(3)=(3,1-2\sqrt{-5})(3,1+2\sqrt{-5})$ where $(3,1-2%
\sqrt{-5}),((3,1+2\sqrt{-5})$ are height one prime ideals and hence maximal (%
$t$-) ideals of $%
\mathbb{Z}
\lbrack \sqrt{-5}].$ (Recall here that a Dedekind domain is a Prufer domain
and so every nonzero ideal of a Dedekind domain is a $t$-ideal.)

Call an integral domain $D$ a $t$-local domain if $D$ is quasi local with
maximal ideal $M$ a $t$-ideal, then $D$ is a HoFD in that every nonzero non
unit of $D$ is a homogeneous element and hence is uniquely expressible as a
product of mutually $t$-comaximal elements. Now every one dimensional local
ring being $t$-local is a HoFD and in view of this fact the following
proposition provides a valuable contrast.

\begin{proposition}
\label{Proposition P3} Let $D$ be an integral domain with each nonzero non
unit a rigid element. Then $D$ is a valuation domain.
\end{proposition}

\begin{proof}
Let $x,y$ be two nonzero non units. Then $xy$ being a nonzero non unit, and
hence rigid, gives $x|y$ or $y|x.$ Thus for every pair of elements we have
one dividing the other.
\end{proof}

It appears that, very few restrictions other than the property $\ast $ will
make semirigid domains into GCD domains. For example a Krull domain is
atomic, and hence a semirigid domain, but not all Krull domains are UFDs. So
some products of rigid elements are not rigid. However the following
statement holds.

\begin{proposition}
\label{Proposition P4} An atomic domain $D$ is a UFD if and only if for
every pair of atoms $a,b$, $(a,b)_{v}\neq D$ implies that $ab$ is rigid.
\end{proposition}

\begin{proof}
We show that if $D$ is atomic such that $(a,b)_{v}\neq D$ implies that $ab$
is rigid for every pair of atoms $a,b,$ then every atom is a prime. For this
take an atom $a$ and some other atom $x$. Note that if $(a,x)_{v}\neq D$
then $a$ and $x$ are associates because, by the condition, $ax$ is rigid and
so $a|x$ or $x|a.$ If $a|x$, then $x=ar$ and because $a$ is a non unit and $%
x $ being an atom $r$ must be a unit, by the definition of an atom.
Similarly if $x|a,$ we conclude that $x$ and $a$ are associates. Taking the
contrapositive we conclude that in an atomic domain with the given property,
if two atoms are non-associates then they are $v$-coprime. Now let $a$ be an
atom and let $x$ be any element in our domain. We claim that in $D,$ $a\nmid
x$ implies that $(a,x)_{v}=D.$ For if $x=x_{1}...x_{n}$ is any atomic
factorization of $x$ and $(a,x)_{v}\neq D$ then $(a,x_{i})_{v}\neq D$ for
some $i$. (Else if $x_{i}$ are all $v$-coprime to $a$, then so is their
product.)

Finally let $r,s$ be any two nonzero elements of $D$ and let $a|rs.$ Then $%
a|r$ or $a|s.$ For if $a\nmid r$ and $a\nmid s$, then $(a,r)_{v}=D$ and $%
(a,s)_{v}=D$ which forces $(a,rs)_{v}=D.$ So the atom that we picked is a
prime. Whence every atom of $D$ is a prime. But an atomic domain in which
every atom is a prime must be a UFD. The converse is obvious.
\end{proof}

Finally, note that if $h$ is a homogeneous element of an integral domain,
then every non unit factor $t$ of $h$ is in $M(h)$ the unique maximal $t$%
-ideal containing $h.$ Thus for each pair of non unit factors $u,v$ of a
homogeneous element $h$ we have $(h,q)_{v}\neq D.$ This leads to the
question: Call a nonzero non unit $q$ of an integral domain $D$ a
pre-homogeneous element if for every pair $r,s$ of non unit factors of $q$
we have $(r,s)_{v}\neq D.$ Must a pre-homogeneous element be homogeneous?

Generally the answer is no, as every rigid element is pre-homogeneous and a
rigid element may not be homogeneous. For example, as we have already
mentioned, an atom is rigid and an atom may belong to more than one maximal $%
t$-ideals. However in some integral domains a pre-homogeneous element may
well be homogeneous. The reason for introducing this new terminology, here,
is that the term "homogeneous" was used for an element, in \cite{[AMZ]}, in
the set up of pre-Schreier domains saying that $h$ is homogeneous if $h$ is
a non unit such that for all non-units $a,b|h$ we have $(a,b)_{v}\neq D.$ At
that time there was no clear concept of homogeneous ideals. This changed in 
\cite{[DZ]} and we called a proper (nonzero) $\ast $-finite $\ast $-ideal $A$
of $D$ homogeneous if $A$ is contained in a unique maximal $\ast $-ideal.
Now Chang \cite{[Ch]} calls an element $h$ a homogeneous element if $hD$ is
a ($t$-) homogeneous ideal and, as we have shown above, this new homogeneous
isn't the old homogeneous. There are two ways of dealing with the situation.
Laugh at Chang for getting a paper out of a name change or smooth things
over, by introducing some new terminology. I have decided to take the latter
approach.

\begin{proposition}
\label{Proposition P5} In a domain $D$ with PSP property, every
pre-homogeneous element is homogeneous.
\end{proposition}

\begin{proof}
Let $q$ be a pre-homogeneous element of the PSP domain $D.$ Suppose that $q$
is not homogeneous. Then there are at least two maximal $t$-ideals $%
M_{1},M_{2}$ containing $q.$ Let $m_{1}\in M_{1}\backslash M_{2}$ and $%
m_{2}\in M_{2}\backslash M_{1}.$ Then $(m_{1},M_{2})_{t}=D$ and $%
(m_{2},M_{1})_{t}=D.$ We can write $(m_{1},M_{2})_{t}=(m_{1},F_{2})_{t}$
where $F_{2}\subseteq M_{2}$ and similarly $%
(m_{2},M_{1})_{t}=(m_{2},F_{1})_{t}$ where $F_{1}$ is a finitely generated
ideal contained in $M_{1}.$ Set $G=(q,F_{1},F_{2},m_{1},m_{2}).$ Now $%
(q,F_{1},m_{1})\subseteq M_{1}$, so $(q,F_{1},m_{1})_{t}\neq D.$ Since $D$
is a PSP domain, $(q,F_{1},m_{1})_{t}\neq D$ means that there is a non unit $%
r$ such that $(q,F_{1},m_{1})\subseteq rD.$ Similarly we can find a nonzero
non unit $s$ in $D$ such that $(q,F_{2},m_{2})\subseteq sD$, because $%
(q,F_{2},m_{2})\subseteq M_{2}.$ Thus $(q,F_{1},F_{2},m_{1},m_{2})\subseteq
(r,s).$ Yet, $(q,F_{1},F_{2},m_{1},m_{2})_{t}=D.$ Whence $(r,s)_{t}=D$ a
contradiction to the assumption that $q$ is pre-homogeneous. Since this
contradiction arises from the assumption that $q$ is contained in more than
one maximal $t$-ideals the conclusion follows.
\end{proof}

\begin{corollary}
\label{Corollary C3} A PSP domain whose nonzero non units are expressible as
finite products of pre-homogeneous elements is a HoFD.
\end{corollary}

Now as we have seen, a rigid element is pre-homogeneous. we have the
following result.

\begin{corollary}
\label{Corollary C4} A semirigid pre-Schreier domain is a HOFD and
consequently a VFD is a HoFD.
\end{corollary}

The proof depends upon the fact that a pre-Schreier domain is PSP, as we
have already seen. Moreover in, a PSP domain and hence, in a pre-Schreier
domain a rigid element is homogeneous. Now use Proposition \ref{Proposition
P1}. Note here that what I have to call pre-homogeneous was once called
homogeneous and according to Theorem 2.3 of \cite{[AMZ]}, a completely
primal pre-homogeneous element is homogeneous.

Yet having established that a VFD is actually a HoFD\ does not change much.
For example we do have an example, in Example \ref{Example E1}, of a
semirigid pre-Schreier domain that is also a HoFD and being a HoFD has no
extra effect on it.

\section{Semi packed domains}

Consider the following conditions satisfied by a nonzero non unit $q$ of $D$:

(a) Call $q$ power rigid if $q$ is a rigid element such that every positive
integral power of $q$ is rigid,

(b) Call $q$ tenacious if $q$ is a completely primal power rigid element.

(d) Call $q$ prime quantum if $q$ is a tenacious element such that for every
non unit factor $h$ of $q$ there is a natural number $n$ with $q|h^{n}.$

(e) Call $q$ a packed element if $q$ is a power rigid element and a packet
i.e.$\sqrt{(q)}$ is a prime ideal.

(f) Call $q$ a t-packed element if $q$ is a tenacious packed element.

Call a domain semi tenacious if every nonzero non-unit of $D$ is expressible
as a finite product of tenacious elements and call $D$ semi packed (resp.,
semi t-packed) if every nonzero non unit of $D$ is expressible as a finite
product of packed (t-packed) elements. Let's adopt the convention of using $%
\sqrt{q}$ to indicate $\sqrt{(q)}$ and of calling a minimal prime of $(x)$ a
minimal prime of $x$. As before, we will call $x,y$ comparable if $x|y$ or $%
y|x.$ With this preparation we begin work on developing the theory of SPDs.

Let's start with the observation that if $q$ is a packed element and $r$ a
non unit factor of $q$, then $r$ is at least a power rigid element.

\begin{lemma}
\label{Lemma L3} Let $q$ be a packed element and let $r$ be a nonzero non
unit such that $\sqrt{q}\subseteq \sqrt{r}.$ Then $r$ and $q$ are
comparable, i.e., $r|q$ or $q|r.$ Consequently, two packed elements $q$ and $%
r$ of a domain $D$ are comparable if and only if $\sqrt{q}$ and $\sqrt{r}$
are comparable.
\end{lemma}

\begin{proof}
Note that $\sqrt{q}\subseteq \sqrt{r}$ implies that $r|q^{n}$ for some $n.$
But $q^{n}$ is rigid. Whence the comparability ($r|q$ or $q|r).$ The
consequently part is obvious.
\end{proof}

\begin{lemma}
\label{Lemma L4}Suppose that a finite product $p_{1}p_{2}...p_{n}$ of packed
elements $p_{1},...,p_{n}$, in a domain $D,$ is such that $\sqrt{%
p_{1}p_{2}...p_{n}}$ is a prime. Then $p_{1}p_{2}...p_{n}$ is packed and $%
p_{i}$ are mutually comparable. Consequently if any pair of $p_{i}$ is
incomparable then $\sqrt{p_{1}p_{2}...p_{n}}$ is not a prime.
\end{lemma}

\begin{proof}
Note that $\sqrt{p_{1}p_{2}...p_{n}}=\cap \sqrt{p_{i}}=\sqrt{p_{j}}$ for
some $j,$ because $\sqrt{p_{1}p_{2}...p_{n}}$ is a prime. Next as $\sqrt{%
p_{j}}\subseteq \sqrt{p_{i}}$ for each $i$ we have $p_{i}|p_{j}^{n_{i}}$ for
some $n_{i}.$ Whence every power of $p_{1}...p_{n}$ divides a power of $%
p_{j},$ forcing $p_{1}...p_{n}$ to be power rigid. That $p_{i}$ are mutually
comparable follows from the fact that $p_{1}...p_{n}$ is power rigid. The
consequently part is obvious.
\end{proof}

\begin{lemma}
\label{Lemma L5}Let $a,b$ be two packed elements of a domain $D$. If $a,b$
are comparable then $ab$ is a packed element.
\end{lemma}

\begin{proof}
Since $a,b$ are comparable $ab|a^{2}$ or $ab|b^{2}.$ Thus $ab$ is power
rigid. Also if $(a)\subseteq (b),$ then $\sqrt{a}\subseteq \sqrt{b}$ and so $%
\sqrt{ab}=\sqrt{a}\cap \sqrt{b}.$ (Same for $(b)\subseteq (a).)$
\end{proof}

Observe that if $a$ and $b$ are packed elements such that $a|b,$ then $a$
divides some power of $b.$ So if $a$ divides no power of $b$, i.e. $\sqrt{b}%
\nsubseteq \sqrt{a}$ then $a\nmid b.$ Thus $a$ and $b$ are incomparable if
and only if $\sqrt{a}$ and $\sqrt{b}$ are incomparable. These observations
lead to the following conclusion.

\begin{proposition}
\label{Proposition P6} Let $D$ be an SPD. Then the following hold. (1).
Every nonzero non unit $x$ of $D$ is expressible as a product of mutually
incomparable packed elements, (2) If $x=x_{1}...x_{n}$ is a product of
mutually incomparable packed elements then the number of minimal primes of $%
x $ is precisely $n.$
\end{proposition}

\begin{proof}
(1). Let $a=p_{1}p_{2}...p_{r}$ be a product of $r$ packed elements. Pick,
say, $p_{1}$ and collect all the factors $p_{i}$ comparable with $p_{1}.$
Using the above lemmas and proceeding exactly as in the proof of Proposition %
\ref{Proposition P1A} we can write $a=a_{1}...a_{n}$ where $a_{i}$ are
mutually incomparable.

(2). Let $x=x_{1}...x_{n}$ be a product of mutually incomparable packed
elements. Taking the radical of both sides we get $\sqrt{x}=\sqrt{%
x_{1}...x_{n}}=\sqrt{x_{1}}\cap \sqrt{x_{2}}\cap ...\cap \sqrt{x_{n}}%
\supseteq $ $\sqrt{x_{1}}\sqrt{x_{2}}...\sqrt{x_{n}}.$ Now any minimal prime
ideal $P$ of $x$ contains $x$ and hence $\sqrt{x_{1}}\sqrt{x_{2}}...\sqrt{%
x_{n}},$ where each of $\sqrt{x_{i}}$ is a prime containing $x.$ Being
minimal, $P$ cannot properly contain prime ideals containing $x.$ Whence $P=%
\sqrt{x_{i}}$ for some $i.$ Also as $P$ has to contain one of the $\sqrt{%
x_{i}}$ we conclude that the number of minimal primes of $x$ is precisely $n.
$
\end{proof}

Question: Is an SPD $D$ pre-Schreier?

To this point we have followed in the footsteps of \cite{[CR]} with good
results. It appears that the main property used in the proof of Proposition
2.1 of \cite{[CR]} is indeed that a valuation element is a packed element.
Thus we have the following result.

\begin{proposition}
\label{Proposition P7} (cf Proposition 2.1of \cite{[CR]}) A semi packed
domain $D$ is pre-Schreier and so an SPD is a semi t-packed domain.
\end{proposition}

Indeed, as remarked earlier, there is no need for a repeat proof. A reader
not quite sure about the proof in \cite{[CR]} may simply assume packed
elements to be tenacious, i.e., completely primal and that'd make an SPD
pre-Schreier. Because the observations that we want to make below depend on
the fact that, defined or proved, an SPD is pre-Schreier.

Having indicated that an SPD is closely related to VFDs it's pertinent to
look for other similarities.

\begin{proposition}
\label{Proposition P7A}A semi tenacious (resp., semi packed) domain is a
HoFD.
\end{proposition}

\begin{proof}
Because an SPD is semi tenacious it suffices to study semi tenacious
domains. Now as a semi tenacious domain $D$ is pre-Schreier we conclude
using Proposition \ref{Proposition P5} that every rigid element and hence
every power rigid element of $D$ is homogeneous and consequently $D$ is a
HoFD.
\end{proof}

As already mentioned, an element $p$ in a domain $D$ is called a packet if $%
(p)$ has a unique minimal prime. That is $\sqrt{p}$ is a prime ideal. This
notion was introduced in \cite{[Z thesis]} while studying GCD rings of Krull
type. Obviously every nonzero non unit of a ring of Krull type has finitely
many minimal primes. It turned out that in a GCD domain $D$ a principal
ideal $xD$ has finitely many minimal primes if and only if $x$ can be
written as a finite product of packets. So a GCD domain was called a Unique
Representation Domain (URD), in \cite{[Z thesis]} if for each nonzero non
unit $x$ of $D$ the ideal $xD$ had finitely many minimal primes. GCD URDs
were presented in \cite{[Z URD]}. A more general study of URDs was carried
out in \cite{[EGZ]} where a general integral domain was called a URD, via 
\cite[Corollary 2.12]{[EGZ]}, if for every nonzero non unit $x$ of $D$ we
can (uniquely) write $xD=(X_{1}X_{2}...X_{n})_{t},$ where $X_{i}$ are
mutually $t$-comaximal $t$-invertible $t$-ideals with $\sqrt{X_{i}}$ prime.
Call $D$ a $t$-treed domain if the set of prime $t$-ideals of $D$ is a tree
under inclusion. Also, using \cite[Corollary 2.12]{[EGZ]}, we concluded that
a URD is $t$-treed.

\begin{lemma}
\label{Lemma L5A} Let $x$ be a nonzero non unit in an SPD. Then $x$ is a
packed element if and only if $x$ is a packet.
\end{lemma}

\begin{proof}
A packed element is a packet anyway. Conversely let $x$ be a packet and
suppose that $x$ is a not a packed element. Since $D$ is semi packed we must
have $x=p_{1}p_{2}...p_{r}$ where $p_{i}$ are mutually incomparable packed
elements of $D.$ But this is impossible unless $r=1,$ because $x$ has a
unique minimal prime.
\end{proof}

Recall that the notion of HoFD was called a semi $t$-pure domain in \cite%
{[AMZ]} not too long ago and, as already noted, that in view of proof of (3)
of Theorem 3.1 of \cite{[AMZ]} one can say that if $D$ is a HoFD and if $%
0\neq xD_{M}\subseteq MD_{M}$ then $xD_{M}\cap D=x^{\prime }D$ where $%
x^{\prime }$ is a homogeneous element contained in $M$ such that $x^{\prime
}D_{M}=xD_{M}.$

\begin{lemma}
\label{Lemma L6} Let $D$ be a HoFD. Then the following hold: (1) every
homogeneous element of $D$ is a packet if and only if $D$ is $t$-treed, (2)
if every homogeneous element of $D$ is a packed element then $D$ is a PVMD
and hence a GCD domain.
\end{lemma}

\begin{proof}
(1) By the proof of Proposition \ref{Proposition P1}, every nonzero non unit 
$x$ of a HoFD can be written as $x=h_{1}h_{2}...h_{r}$ where $h_{i}$ are
mutually $v$-coprime ($t$-comaximal) homogeneous elements. If each
homogeneous element is a packet then \cite[Corollary 2.12]{[EGZ]} applies,
so $D$ is a URD and hence $t$-treed. Conversely if $D$ is $t$-treed and a
HoFD then every homogeneous element of $D$ has a unique minimal prime. Next,
for (2), note that in a HoFD every nonzero non unit $x$ can be written as $%
x=x_{1}x_{2}...x_{n}$ where $x_{i}$ are mutually $t$-comaximal homogeneous.
Since each homogeneous element is packed and hence a packet, $D$ is $t$%
-treed by (1). Note that $D$ is a HoFD to start with. Thus if $M$ is a
maximal $t$-ideal of $D$ and $0\neq mD_{M}\subseteq MD_{M}$, then $%
mD_{M}\cap D=m^{\prime }D$ where $m^{\prime }$ is a homogeneous element of $%
D $ in $M$ such that $m^{\prime }D_{M}=mD_{M}$. Now let $0\neq
hD_{M},kD_{M}\subseteq MD_{M}.$ Then correspondingly $h^{\prime },k^{\prime
} $ are homogeneous elements belonging to $M$. Because $D$ is $t$-treed and
because homogeneous elements are packed we conclude that $\sqrt{h^{\prime }}$%
and$\sqrt{k^{\prime }}$ are comparable. But then so, by Lemma \ref{Lemma L3}%
, are $h^{\prime }$ and $k^{\prime }$ comparable in $D$ and consequently $%
h^{\prime }D_{M},k^{\prime }D_{M}$ in $D_{M}.$ Since $h^{\prime
}D_{M}=hD_{M} $ and $k^{\prime }D_{M}=kD_{M}$, we have $hD_{M}\subseteq
kD_{M}$ or $kD_{M}\subseteq hD_{M},$ forcing $D_{M}$ to be a valuation
domain. Now as $M$ is a typical maximal $t$-ideal, $D_{M}$ is a valuation
domain for each maximal $t$-ideal of $D.$ Thus $D$ is a PVMD. Again as $D$
is a HoFD, $Cl_{t}(D)=0$ and a PVMD $D$ with $Cl_{t}(D)=0$ is a GCD domain,
as noted in the introduction.
\end{proof}

\begin{theorem}
\label{Theorem T2} TFAE for a $t$-treed domain $D.$ (1) $D$ is a VFD, (2) $D$
is an SPD, (3) $D$ is a semirigid GCD domain, (4) $D[X]$ is a VFD and (5) $%
D[X]$ is semi packed.
\end{theorem}

\begin{proof}
(1) $\Rightarrow $ (2). This follows because a VFD is semi packed.

(2) $\Rightarrow $ (3). Note that an SPD is a HoFD and that in a HoFD every
nonzero non unit $x$ can be written as $x=x_{1}x_{2}...x_{n}$ where $x_{i}$
are mutually $t$-comaximal homogeneous. Now as $D$ is a HoFD being semi
packed and treed all homogenous elements are packets, by (1) of Lemma \ref%
{Lemma L6}. Also since $D$ is semi packed, every packet is a packed element
by Lemma \ref{Lemma L5A}. Whence by (2) of Lemma \ref{Lemma L6} $D$ is a GCD
domain. Now as a HoFD GCD domain is a GCD weakly Matlis domain which is a
semirigid domain.

(3) $\Rightarrow $ (4). This follows because if $D$ is a semirigid GCD
domain then so is $D[X]$ \cite{[Z semi]} and as indicated in the previous
section a semirigid GCD domain is a VFD.

(4) $\Rightarrow $ (5). This direct because a VFD is semi packed.

(5) $\Rightarrow $ (1). Since $D[X]$ being semi packed entails $D[X]$ being
pre-Schreier and $D[X]$ being pre-Schreier requires $D$ to be integrally
closed, by Corollary 10 of \cite{[ADZ]}. Thus $D$ is Schreier. Also $D[X]$
being of finite $t$-character forces $D$ to be of finite $t$-character, \cite%
[Corollary 3.3]{[ACZ]}. Again $D[X]$ being semi packed means ($D[X]$ is a
HoFD and so) no two maximal $t$-ideals of $D[X]$ contain a nonzero prime
ideal. But maximal $t$-ideals of $D[X]$ are either uppers to zero or of the
form $P[X]$ where $P$ is a maximal $t$-ideal. This leads to the conclusion
that between any two maximal $t$-ideals $P,Q$ of $D$ there is no nonzero
prime ideal. Combining this piece of information with $D$ being of finite $t$%
-character we conclude that $D$ is a weakly Matlis domain. Noting also that $%
Cl_{t}(D)=0$ because $D$ is Schreier, we conclude that $D$ is a HoFD.

Moreover $D$ being $t$-treed and of finite $t$-character forces $D$ to be a
URD with every nonzero non unit element a product of mutually $t$-comaximal
packets \cite[Corollary 2.12]{[EGZ]}. Now, all we need show is that each
packet of $D$ is a packed element. For this let $p$ be a packet of $D.$
Indeed $p$ is a packet in $D[X]$ and hence a packed element by Lemma \ref%
{Lemma L5A}, because $D[X]$ is semi packed. But then $p$ is a packed element
of $D.$ Thus $D$ is semi packed. Next using the steps taken in the proof of
(2) $\Rightarrow $ (3) we can show that $D$ is a GCD domain. Now $D$ is a
GCD domain and a GCD HoFD is a weakly Matlis GCD domain and hence a $t$%
-treed VFD.
\end{proof}

The above theorem shows just how close semi packed domains are to VFDs,
without being integrally closed. SPDs have at least one advantage over VFDs,
we have at least one example of a genuine semi packed domain that is not a
semirigid GCD domain.

\begin{example}
\label{Example E2} The ring $D=%
\mathbb{Z}
_{(p)}+(X,Y)%
\mathbb{Q}
\lbrack \lbrack X,Y]]$ in Example \ref{Example E1} is precisely an example
of an SPD.
\end{example}

Illustration: An element of $D$ is either $\epsilon p^{\alpha },$ where $%
\epsilon $ is a unit of $D$ and $\alpha $ a non-negative integer or of the
form $f/p^{\beta }$ where $f$ is a non unit of $%
\mathbb{Q}
\lbrack \lbrack X,Y]]$ and $\beta $ an integer. As in Example \ref{Example
E1} , it is easy to check that if $f$ is a prime of $%
\mathbb{Q}
\lbrack \lbrack X,Y]],$ pairs of factors of $($ $f/p^{\beta })^{n}$ are
comparable for each positive integer $n.$ Coupling this piece of information
with the fact that for $f$ a prime in $%
\mathbb{Q}
\lbrack \lbrack X,Y]],$ $($ $f/p^{\beta })(%
\mathbb{Q}
\lbrack \lbrack X,Y]])$ is a height one prime we conclude that $($ $%
f/p^{\beta })$ has a unique minimal prime and so is a packed element.
Finally as every nonzero non unit of $%
\mathbb{Q}
\lbrack \lbrack X,Y]]$ is of the form $(f_{1})^{n_{1}}...(f_{r})^{n_{r}}$,
where $f_{i}$ are primes and $n_{i}$ non-negative integers, we conclude that
every nonzero non unit of $D$ is of the form $\epsilon p^{\alpha
}(f_{1}/p^{\beta _{1}})^{n_{1}}...(f_{r}/p^{\beta _{1}})^{n_{r}},$ $\alpha
,\beta _{i},n_{i}\in 
\mathbb{Z}
,\alpha ,n_{i}\geq 0$ and hence a product of packed elements. Of course we
already know that $D$ is Schreier.

The above example establishes the existence of an SPD that is not a
semirigid GCD domain. (Or actually the existence of a HoFD in which not all
homogeneous elements are packed elements.) On the other hand, being
integrally closed, it raises the question: Is the ring in Example \ref%
{Example E2} a VFD? Then there is the question: When is a VFD (resp., semi
packed domain, semi tenacious domain) a semirigid GCD domain? We list in
propositions below some of the "off the cuff" answers to the question: When
is a semirigid domain a semirigid GCD domain. Some of these "answers" may
address the specific questions about VFDs etc.

\begin{theorem}
\label{Theorem T3}TFAE for an integral domain $D.$ (1) $D$ is a semi rigid
GCD domain, (2) $D$ is a GCD weakly Matlis domain, (3) $D$ is a PVMD weakly
Matlis domain with $Cl_{t}(D)=0,$ (4) $D$ is an independent ring of Krull
type with $Cl_{t}(D)=0,$ (5) $D$ is an SPD in which products of pairs of non
coprime packed elements are packed elements, (6) $D$ is an SPD whose
homogeneous elements are all packed elements, (7) $D$ is a semi packed URD,
(8) $D$ is a VFD URD, (9) $D$ is a pre-Schreier semirigid domain that is
also a URD in which $r$ is rigid $\Leftrightarrow $ $r$ is a packet, (10) $D$
is a PSP semirigid domain in which products of pairs of non coprime rigid
elements are again rigid, (11) $D$ is an SPD such that for every pair $x,y$
of packed elements we have $(x,y)_{v}=D$ or one or both of the following
hold $(x,y)_{v}=xD$ or $(x,y)_{v}=yD$, (12) $D$ is a semi rigid domain and
for every rigid element $r$ the following holds: For all $x\in D\backslash
\{0\},$ $(r,x)_{v}=hD$ where $h$ is a factor of $r$ and (13) $D$ is a VFD
such that every valuation element of $D$ has a GCD with every nonzero
element of $D.$
\end{theorem}

\begin{proof}
(1) $\Leftrightarrow $ (2). This is (1) $\Leftrightarrow $ (5) of Corollary %
\ref{Corollary C2}.

(2) $\Leftrightarrow $ (3). This is obvious because a PVMD $D$ with $%
Cl_{t}(D)=0$ is a GCD domain.

(1) $\Leftrightarrow $ (4). This follows from Corollary 3.8 of \cite{[AMZ]}.

(1) $\Leftrightarrow $ (5). It follows from Corollary \ref{Corollary C2} in
that an SPD is a semirigid domain.

(5)$\Rightarrow $ (6). This follows from the fact that two homogeneous
elements belonging to the same maximal $t$-ideal are non-coprime and packed
elements in an SPD are homogeneous.

(6) $\Rightarrow $ (2). This follows from using the fact that an SPD is a
HoFD and using (2) Lemma \ref{Lemma L6}. Now a GCD HoFD is a GCD weakly
Matlis domain.

(1) $\Rightarrow $ (7). For this note that a semirigid GCD domain is $t$%
-treed and of finite character and so must be a URD, also being a GCD domain
every rigid element in $D$ is power rigid and as we have shown that $D$ is a
URD every rigid element in $D$ is a packed element.

(7)$\Rightarrow $ (1). This follows from taking the following steps: (i)
noting that an SPD is a HoFD and that in an SPD a packet is a packed element
(Lemma \ref{Lemma L5A}) (ii) using (2) of Lemma \ref{Lemma L6} and (iii)
noting that a GCD\ HoFD is a "semi $t$-pure" GCD domain and hence a
semirigid GCD domain by Corollary 3.8 of \cite{[AMZ]}.

(1) $\Rightarrow $ (8). This follows from the fact that every rigid element $%
r$ in a semirigid GCD domain is a valuation element. (For the rigid element $%
r$ show as in the proof of (2) $\Rightarrow $ (3) of Corollary \ref%
{Corollary C2} that $P(r)=\{x\in D|(x,r)_{v}\neq D\}$ is a maximal $t$-ideal
of $D$ and establish that $rD_{P(r)}\cap D=rD,$ using the fact that $r$ is
homogeneous and that $D_{P(r)}$ is a valuation domain.) That a semirigid GCD
domain is a URD is now immediate. (8) $\Rightarrow $ (7). This is direct
since a VFD is semi packed.

(1) $\Rightarrow $ (9). This is direct since a semirigid GCD domain is a
semirigid pre-Schreier domain, and a URD in which $r$ is rigid $%
\Leftrightarrow $ $r$ is a packet.

(9) $\Rightarrow $ (1). This follows from the fact that $D$ is a
pre-Schreier domain in which every nonzero non unit is expressible as
product of mutually coprime packets each of which is rigid. (One may look at
the product of two non-coprime rigid elements $r,s.$ Then $%
rs=p_{1}p_{2}..,.p_{n}$ where $p_{i}$ are mutually coprime packets. Now $%
r|p_{1}p_{2}..,.p_{n}$ implies $r=r_{1}r_{2}..,.r_{n}$ where $r_{i}|p_{i}.$
But as $r$ is rigid and so cannot have mutually coprime non unit factors, we
conclude that $r|p_{i}$ for exactly one $i,$ say $r|p_{1}.$ Similarly $s$
divides exactly one of $p_{i}$ and that cannot be other than $p_{1}.$ So $s$
is coprime with $p_{2},...p_{n}.$ Now $s=(p_{1}/r)p_{2}...p_{n}$ and $%
s,p_{i} $ coprime for $i\geq 2$ forces $s$ to divide ($p_{1}/r)$. But then
it is easy to conclude that $rs=p_{1}.$ Thus in $D$ the property $\ast $
holds and \ Corollary \ref{Corollary C2} applies.)

(1) $\Rightarrow $ (10). This is direct as a semirigid GCD domain is a PSP
domain in which products of non-coprime rigid elements are rigid.

(10) $\Rightarrow $ (1). This follows because $D$ has PSP and so "products
pairs of non-coprime rigid elements being rigid" implies "products of pairs
of non-$v$-coprime rigid elements being rigid" which is property $\ast $ and
Corollary \ref{Corollary C2} applies.

(1) $\Rightarrow $ (11). This follows because in a semirigid GCD domain
every rigid element is a valuation element, as shown in the proof of (1) $%
\Rightarrow $ (8), and hence a packed element. Moreover for a pair of rigid
elements $r,s$ in a semirigid GCD domain $(r,s)_{v}=D$ or $r$ and $s$ are
comparable.

(11) $\Rightarrow $ (1). A semi packed domain with the given condition is
clearly a semirigid pre-Schreier domain in which every nonzero non unit is a
product of mutually coprime rigid elements and so the proof of (9) $%
\Rightarrow $ (1) applies.

(1) $\Rightarrow $ (12). This follows directly as a semirigid GCD domain is
a GCD domain.

(12) $\Rightarrow $ (1). (Use proof of Theorem \ref{Theorem T1} or proceed
as follows.) For this we first note that every rigid element $r$ of $D$ is
primal. This is because letting $r|xy$ we get $r=r_{1}r_{2}$ where $%
r_{1}=(r,x)_{v}.$ So $r=r_{1}r_{2}$ and $x=r_{1}x_{2}$ where $%
(r_{2},x_{2})_{v}=D.$ Now as $r_{2}|x_{2}y$, and $(r_{2},x_{2})_{v}=D,$ we
conclude that $r_{2}|y.$ Using the facts that $D$ is semirigid, that
products of primals are primal and that units are primal we conclude that
all nonzero elements of $D$ are primal and that $D$ is pre-Schreier. But
then every rigid element is homogeneous (i.e. belongs to a unique maximal $t$%
-ideal). Once that done, note that $(r,x)_{v}=(r,x)_{t}.$ Now with some
effort one can show that, in the terminology of \cite{[AZ starsh]}, every
rigid element of $D$ is $t$-f-homog and $D$ a $t$-f-Semi Homogeneous ($t$%
-f-SH) domain. Setting $\ast =t$ in Theorem 17 of \cite{[AZ starsh]} we
conclude that $D$ is a GCD independent ring of Krull type and hence a
semirigid GCD domain.

(1) $\Rightarrow $ (13). This is direct as a semirigid GCD domain is a VFD
with rigid elements being the valuation elements with all those properties.

(13) $\Rightarrow $ (12). This is obvious because in a VFD valuation
elements are rigid and every valuation element having GCD translates to $%
(r,x)_{v}=hD$ because a VFD is Schreier and in a Schreier domain coprime is $%
v$-coprime (Lemma 2.1 of \cite{[Z wb]}).
\end{proof}

While the jury is out on whether a VFD is a semirigid GCD domain or not, it
is patent that generally a semi tenacious domain is not a semirigid GCD
domain, nor is an SPD a semirigid GCD\ domain. However each collapses into a
semirigid GCD domain if $D$ is of $t$-dimension one, i.e., if every maximal $%
t$-ideal of $D$ is of height one. Yet before we show that a bit of
introduction is in order.

Call a domain $D$ a generalized UFD (GUFD) if every nonzero non unit of $D$
is expressible as a finite product of prime quanta. It was shown in \cite{[Z
thesis]} and later in \cite{[AAZ1]}, that every nonzero non unit of a GUFD
is uniquely expressible as a finite product of mutually coprime prime
quanta. We showed a GUFD to be a GCD domain that was also a generalized
Krull domain (GKD), that is a domain $D$ such that (a) $D$ is a locally
finite intersection of localizations at all height one primes of $D$ and (b) 
$D_{P}$ is a valuation domain for each height one prime of $D.$ It so
transpired that later, in \cite{[AMZ fin]}, domains with just the (a) part
were studied as weakly Krull domains. A weakly factorial domain $D$ is a
weakly Krull domain with $Cl_{t}(D)$ trivial. And as noted on page 350 of 
\cite{[AMZ]}, just above Corollary 3.8, a weakly Krull domain that is a GCD
domain is a GUFD.

\begin{proposition}
\label{Proposition P8} The following are equivalent for an integral domain $%
D:$ (1) $D$ is a weakly factorial GCD domain, (2) $D$ is VFD of $t$%
-dimension $1$, (3) $D$ is a generalized Krull GCD domain, (4) $D$ is a
GUFD, (5) $D$ is a weakly Krull GCD domain, (6) $D$ is a semi tenacious
domain of $t$-dimension $1,$ (7) $D$ is semi packed with $t$-dim($D)=1$ (8) $%
D$ is a semirigid GCD domain with $t$-dimension one.
\end{proposition}

\begin{proof}
(1) $\Leftrightarrow $ (2) follows from \cite[Corollary 1.9]{[CR]}. Next (1) 
$\Leftrightarrow $ (3) $\Leftrightarrow $ (4) follow from Theorem 10 of \cite%
{[AAZ1]}. (3) $\Rightarrow $ (5) because a generalized Krull domain is a
weakly Krull domain and (5) $\Rightarrow $ (3) follows from the observation
that $D$ weakly Krull implies that $D=\cap _{P\in X^{1}(D)}D_{P}$ where the
intersection is locally finite. Next if $D$ is GCD and $P$ of height one,
then $D_{P}$ is a rank one valuation. But then $D$ is a generalized Krull
domain. Now (4) $\Rightarrow $ (6) because a GUFD is weakly Krull and hence
of $t$-dimension one via \cite[Lemma 2.1]{[AMZ fin]}. Next, by definition, a
GUFD is tenacious and thus tenacious of $t$-dimension one. Next comes (6) $%
\Rightarrow $ (4). By definition a semi tenacious domain is a pre-Schreier
domain and being power rigid every tenacious element is pre-homogeneous. But
by Proposition \ref{Proposition P5} every tenacious element is homogeneous.
Let $h$ be a tenacious element in the semi tenacious domain of $t$-dimension
one and let $P(h)=\{x\in D|(h,x)_{v}\neq D|$ be the maximal $t$-ideal
containing $h.$ Now let $Q$ be the minimal prime of $h.$ Then $Q$ is of
height one and hence a maximal $t$-ideal, because $D$ is of $t$-dimension
one. Thus $\sqrt{h}=P(h).$ Finally let $k$ be a non unit factor of $h.$ Then 
$k\in P(h)$ and $\sqrt{k}=P(k)=P(h).$ Consequently $h|k^{n}$ for some
natural number $n.$ But this forces $h$ to be a prime quantum. Thus every
tenacious element is a prime quantum. Whence $D$ is a GUFD. Next, as a
packed element is tenacious we have (7) $\Rightarrow $ (6) and we have
already shown that (6) $\Rightarrow $ (4). Finally for (4) $\Rightarrow $
(7) note that a GUFD is of $t$-dimension one and semi packed. Finally (4) $%
\Rightarrow $ (8) because a GUFD is semirigid GCD and of $t$-dimension one.
For (8) $\Rightarrow $ (4), note that if $r$ is a rigid element in a
semirigid GCD domain of $t$-dimension one and if $h,k$ are non unit factors
of $r^{n}.$ Then $h$ and $k$ are non-coprime. (For if they were coprime then 
$GCD(h,r)$ and $GCD(k,r)$ would be coprime. Forcing $GCD(h,r)=1$ or $%
GCD(k,r)=1.$ Yet neither would be possible as $h,k|r^{n}.)$ Thus every pair
of non unit factor $r^{n}$ is comparable and $r$ is power rigid. But that
makes $r$ tenacious. Now complete the proof as in the proof of (6) $%
\Rightarrow $ (4).
\end{proof}

Let's face it, factorization started with humans playing with natural
numbers and finding out that if we factorize a random natural number $n$,
i.e. write $n$ as a product of two numbers there's a possibility that the
only way we can do that is $n=1.n=n.1.$ Such numbers that were different
from one got to be called irreducible if they "could not be factored
further" i.e. could not be written as $n=pq$ where $p$ and $q$ are natural
numbers different from $1.$ These days irreducible elements of this ilk are
called atoms. (Though primes were born, even before Democritus put his knife
to his apple to theorize about "atoms".) But rigid elements, packets packed
elements and prime quanta are different kinds of beasts, in that they would
keep on dividing as many times as you want without an end in sight and study
of these beasts was initiated in my doctoral thesis \cite{[Z thesis]}. A
fact the authors of \cite{[CR]} worked so hard to cloak, with the help of a
pre-arranged referee.

\end{document}